\documentclass{amsart}

\usepackage{graphicx}
\usepackage{tikz-cd}
\usepackage{hyperref}

\newtheorem{theorem}{Theorem}[section]
\newtheorem{lemma}[theorem]{Lemma}
\newtheorem{corollary}[theorem]{Corollary}

\theoremstyle{definition}
\newtheorem{definition}[theorem]{Definition}
\newtheorem{example}[theorem]{Example}

\theoremstyle{remark}

\numberwithin{equation}{section}

\newcommand{\rightdoublearrow}{\arrow[r, shift right=0.6ex]\arrow[r, shift left=0.6ex]}
\newcommand{\leftquadruplearrow}{\arrow[l, shift right=1.2ex]\arrow[l, shift left=0.4ex]\arrow[l, shift right=0.4ex]\arrow[l, shift left=1.2ex]}
\newcommand{\righttwicedoublearrow}{\arrow[rr, shift right=0.6ex]\arrow[rr, shift left=0.6ex]}
\newcommand{\lefttwicedoublearrow}{\arrow[ll, shift right=0.6ex]\arrow[ll, shift left=0.6ex]}
\newcommand{\lefttwicetriplearrow}{\arrow[ll, shift right=0.8ex]\arrow[ll]\arrow[ll, shift left=0.8ex]}
\begin{document}

\title{Maximal green sequences of quivers with multiple edges}

\author{Kiyoshi Igusa}
\address{Department of Mathematics, Brandeis University, Waltham, Massachusetts 02453}
\curraddr{Department of Mathematics, Brandeis University, Waltham, Massachusetts 02453}
\email{igusa@brandeis.edu}
\author{Ying Zhou}
\address{Department of Mathematics, Brandeis University, Waltham, Massachusetts 02453}
\curraddr{Department of Mathematics, Brandeis University, Waltham, Massachusetts 02453}
\email{yzhou935@brandeis.edu}


\date{Feb 2019.}


\begin{abstract}
In this paper we completely describe maximal green sequences (MGS) of acyclic quivers with multiple edges in terms of maximal green sequences of their multiple edge-free (ME-free) versions. In particular we establish that any MGS of a quiver must be an ME-free MGS of its ME-free version. For quivers with oriented cycles our results for acyclic quivers do not apply. However a weaker result is available under certain conditions.\\
\end{abstract}

\maketitle
\section{Introduction}
\indent Maximal green sequences (MGSs) were invented by Bernhard Keller \cite{Kel11}. Brustle-Dupont-Perotin \cite{BDP13} and the paper by the first author together with Brustle, Hermes and Todorov \cite{BHIT15} have proven that there are finitely maximal green sequences when the quiver is of finite, tame type or the quiver is mutation equivalent to a quiver of finite or tame types. Furthermore in \cite{BHIT15} it is proven that any tame quiver has finitely many $k$-reddening sequences.\\
\indent However the situation is still pretty much uncharted in the wild case other than cases where the quiver has three vertices which was proven in  \cite{BDP13} which contains a proof highly dependent on the quiver only having three vertices. Despite the fact that the wild case is still unknown in general we can indeed solve it for many easy cases. For example for quivers such as the $k$-Kronecker quiver and $\begin{tikzcd}1\arrow[r, shift right=0.6ex]\arrow[r, shift left=0.6ex] & 2\arrow[r, shift right=0.6ex]\arrow[r, shift left=0.6ex] & 3\end{tikzcd}$ things are really simple due to the Target before Source Theorem in \cite{BHIT15}.\\
\indent In this paper we will generalize the results and introduce three theorems that can significantly simplify understanding of maximal green sequences in simply-laced quivers with multiple edges.\\
\indent We can completely describe MGSs of ME-ful quivers using MGSs of their ME-free versions.\\
\begin{theorem}
(Theorem \ref{T1B}) MGSs of an acyclic quiver $Q$ are a subset of the set of $Q$-ME-free MGSs of its ME-free version, $Q'$.\label{T1}
\end{theorem}
\begin{theorem}
(Theorem \ref{T3B}) Let $Q$ be an ME-ful acyclic quiver and $Q'$ be its ME-free version. The MGSs of $Q$ are exactly the $Q$-ME-free MGSs $(C_0,C_1,\cdots C_m)$ of $Q'$ such that for any multiple edge from $i$ to $j$ in $Q$ for any $C$-matrix $C_i$ in the MGS such that there exists a negative $c$-vector with support containing $i$ the mutation on $C_i$ in the MGS isn't done on any negative $c$-vector with support containing $j$.\label{T3}
\end{theorem}
\indent Let $Q$ be an MGSs of an ME-ful acyclic quiver and $Q'$ be its ME-free version. The MGSs of $Q$ are exactly the $Q$-ME-free MGSs $(C_0,C_1,\cdots C_m)$ of $Q'$ such that for any multiple edge from $i$ to $j$ in $Q$ for any $C$-matrix $C_i$ in the MGS such that there exists a negative $c$-vector with support containing $i$ the mutation on $C_i$ in the MGS isn't done on any negative $c$-vector with support containing $j$.
\indent We can obtain the following crucial corollaries in the acyclic case:
\begin{corollary}\label{C}
(Corollary \ref{CB})The following statements are true:
\begin{enumerate}
\item The number of maximal green sequences of a quiver $Q$ is no greater than that of its ME-free version.
\item All quivers with an MGS-finite ME-free version must themselves be MGS-finite.
\item No minimally MGS-infinite quiver can contain multiple edges.
\item Any two ME-equivalent quivers are MGS-equivalent to each other.
\end{enumerate}
\end{corollary}
\indent If the quiver isn't necessarily acyclic we still have the following result:
\begin{theorem}
(Theorem \ref{T2B})Assume that ($\tilde{Q},\breve{Q})$ are $k$-partition of $Q$ for some $k>1$ any MGS of $Q$ is an MGS of $\tilde{Q}\cup\breve{Q}$.\label{T2}
\end{theorem}
\indent In Section 2 we will provide the background required to understand the rest of the paper. In Section 3 we will prove Theorems \ref{T1} and \ref{T3}. In Section 4 we will prove Theorem \ref{T2}.
\section{Background}
\subsection{Quiver Mutations and Maximal Green Sequences}
\indent The concept of maximal green sequences has many different equivalent definitions. We will use a simple definition using quiver mutations. A \textit{cluster quiver} is a quiver without loops or 2-cycles. Mutations of cluster quivers at vertex $k$ are defined in the following way:
\begin{enumerate}
\item For any pair of arrows $i\to k$ and $k\to j$ add an arrow $i\to j$.
\item Reverse all arrows starting from or ending up in $k$.
\item Delete all 2-cycles that are formed due to process (1) and (2).
\end{enumerate}
\begin{definition}
\begin{enumerate}
\item An \textit{ice quiver} is a quiver $Q$ where a possibly empty set, $F\subseteq Q_0$, consists of vertices that can not mutate.
\item The \textit{framed quiver} $\hat{Q}$ of $Q$ is obtained from $Q$ by adding a vertex $i'$ and an arrow $i\rightarrow i'$ for every $i\in Q$. It is an ice quiver in which all the added vertices can not mutate.
\item The \textit{coframed quiver} $\breve{Q}$ of $Q$ is obtained from $Q$ by adding a vertex $i'$ and an arrow $i'\rightarrow i$ for every $i\in Q$. It is also an ice quiver in which all the added vertices can not mutate.
\end{enumerate}
\end{definition}
\indent An ice quiver $(Q,F)$ can not mutate at elements of $F$, so we call them \textit{frozen vertices}.
\begin{definition}
A non-frozen vertex $i$ is \textit{green} if and only if no arrow from a frozen vertex to $i$ exists. Otherwise it is \textit{red}.\cite{Kel11}
\end{definition}
$\begin{tikzcd}
1 \arrow[r] \arrow[green]{d} & 2\arrow[green]{d}\arrow[r,"\mu_1", shift right=3.5ex]  & 1&2\arrow[l]\arrow[green]{d}\\
1' & 2'&1'\arrow[red]{u}& 2'\\
\end{tikzcd}$
\begin{definition}
A \textit{green sequence} is a sequence $\mathbf{i}=(i_1, i_2,\cdots, i_N)$ such that for all $1\leq t\leq N$ the vertex $i_t$ is green in the partially mutated ice quiver $\hat{Q}(\mathbf{i},t)=\mu_{i_{t-1}}\cdots\mu_2\mu_1(\hat{Q})$.
\end{definition}
\begin{definition}
A \textit{maximal green sequence} is a green sequence such that $\hat{Q}(\mathbf{i},N)$ does not have any green vertices.
\end{definition}
\begin{example} For quiver $1\to 2$ Here is one of its two maximal green sequences.
 $\begin{tikzcd}
1 \arrow[r] \arrow[d] & 2\arrow [d]\arrow[r,"\mu_1", , shift right=3.5ex]  & 1&2\arrow[l]\arrow[d]\arrow[r,"\mu_2", shift right=3.5ex]&1\arrow[r] & 2\\
1' & 2'&1'\arrow[u]& 2'&1'\arrow[u] & 2'\arrow[u]\\
\end{tikzcd}$
\end{example}
\subsection{$c$-vectors and exchange matrices}
\indent We can also use $c$\textit{-vectors} for this purpose. To do so we need to reinterpret mutations of cluster quivers in terms of mutations of matrices. We recall that cluster quivers correspond to \textit{exchange matrices} as defined below. For more details we recommend \cite{FZ01} and \cite{FZ06}.\\
\begin{definition}
\cite{FZ01} An \textit{exchange matrix} of a cluster quiver $Q$ with $n$ vertices is an $n\times n$ matrix such that $b_{ij}$ is the number of arrows from $i$ to $j$ minus the number of arrows from $j$ to $i$.
\end{definition}
\indent It is easy to see that exchange matrices of cluster quivers are always antisymmetric which is not true in the more general case of \textit{valued quivers} which we won't discuss in this paper. Moreover there is a 1-1 correspondence between antisymmetric exchange matrices and cluster quivers.\\
\indent Mutations of exchange matrices are defined here which exactly agree with mutations of cluster quivers.\\
\begin{definition}
\cite{FZ01} If we mutate an $n\times n$ exchange matrix $B = (b_{ij})$ at $k$ we obtain $B' = (b'_{ij})$ defined here.
$b'_{ij} = \begin{cases}
-b_{ij} & \text{if }i = k\text{ or }j = k\\
b_{ij} + b_{ik}|b_{kj}| & \text{if }b_{ik}b_{kj} > 0\\
b_{ij} & \text{in all other cases}
\end{cases}$
\end{definition}
\indent Each partially mutated ice quiver corresponds to an \textit{extended exchange matrix} defined below.\\
\begin{definition}
The \textit{extended exchange matrix} $B'$ corresponding to a partially mutated ice quiver $Q'$ is an $2n\times n$ matrix with the rows corresponding to vertices $\{1,2,\cdots, n, 1', 2',\cdots n'\}$ while the columns corresponds to the vertices $\{1,2,\cdots, n\}$. Here we use the number $n+i$ to represent $i'$. $b_{ij}$ is the number of arrows from $i$ to $j$ minus the number of arrows from $j$ to $i$.
\end{definition}
\indent An extended exchange matrix $B'$ has an upper and lower square submatrices, $B$ and $C$ respectively. The lower square matrix $C$ is known as the $C$\textit{-matrix}. Column vectors of an $C$-matrix are known as $c$\textit{-vectors}. \\
\indent We can use a sequence of $c$-vectors to denote a maximal green sequence because we can use the $c$-vector corresponding to vertex $k$ to represent mutation at vertex $k$.\\
\indent In the following theorem by the first author an equivalent definition of maximal green sequences was introduced. To understand more about the wall-and-chamber structure we suggest that the reader reads \cite{IOTW15}, \cite{GHKK14} or \cite{BST17}.
\begin{theorem}
\cite{I17} Let $\Lambda$ be a finite dimensional hereditary algebra over a field $K$. Let $\beta_1,\cdots,\beta_m\in \mathbb{N}^n$ be any finite sequence of nonzero, nonnegative integer vectors. Then the following are equivalent.
\begin{enumerate}
\item[(a)] There is a generic green path $\gamma:\mathbb{R}\to\mathbb{R}^n$ which crosses the walls $D(M_i)$, $i=1,\cdots,m$ in that order, and no other walls, so that $\dim M_i=\beta_i$ for all $i$.
\item[(b)] There is a maximal green sequence for $\Lambda$ of length $m$ whose $i$th mutation is at the $c$-vector $\beta_i$. 
\end{enumerate}
\end{theorem}
\indent Positive $c$-vectors are dimension vectors of elements of simple-minded collections. Such elements are all bricks. That is, all $c$-vectors are Schur. However we can indeed prove more. They are in fact real as well.\\
\begin{lemma}
Let $k$ be an algebraically closed field. Let $\Lambda$ be a hereditary algebra over $k$. Then any $c$-vector $c$ that appears in any MGS is a real Schur root.
\end{lemma}
\begin{proof}
\indent Since the simples of $\Lambda$ are all exceptional if the lemma were incorrect then there must be some $c$-matrix in the MGS, $C$ such that all columns of $C$ are real Schur roots while one green mutation can somehow generate a root that isn't real. Here there can only be two cases, namely some mutation performed on $-v$ caused some $-w$ to be transformed into $-w'=-w-kv$ which isn't real, some mutation performed on $-v$ caused some $+w$ to be transformed into $w'=w-kv$ which isn't real. In the second case $w'$ may be positive or negative.\\
\indent For an arbitrary $c$-vector $v$ let $M_v$ be the brick such that $v$ is the dimension vector of $M_v$. In this case $\langle v,v\rangle=1-dim Ext^1(M_v,M_v)$. Hence $v$ being real is equivalent to $\langle v,v\rangle=1$.\\
\indent \textbf{Case 1}: Assume that some mutation performed on $-v$ caused some $-w$ to be transformed into $-w'=-w-kv$ which isn't real. $\langle w', w'\rangle = \langle w, w\rangle + k\langle v,w\rangle + k\langle w,v\rangle + k^2\langle v,v\rangle$. Since $Hom(M_v, M_w) = Hom(M_w, M_v) = 0$ due to $M_v$, $M_w$ being two elements in a simple-minded collection and $v, w$ are both real  $\langle w', w'\rangle = k^2+1-k\, dim Ext^1(M_v, M_w) - k\, dim Ext^1(M_w, M_v)$. Using properties of simple-minded collections $dim Ext^1(M_w, M_v)  = k$. Using Prop 6.4 in \cite{KQ15} we can see that $Ext^1(M_v, M_w)  = 0$, Hence $\langle w', w'\rangle = 1$. $w'$ is real.\\
\indent \textbf{Case 2}: Assume that some mutation performed on $-v$ caused some $w$ to be transformed into $-w'=w-kv$ which isn't real. Using properties of simple-minded collections it is obvious that $Ext^1(M_v, M_w) = Hom(M_v, M_w) = 0$. Regardless of whether $w'$ is positive or negative $\langle w', w'\rangle = \langle w, w\rangle - k\langle v,w\rangle 0 k\langle w,v\rangle + k^2\langle v,v\rangle = k^2+1 -k\,dim Hom(M_w, M_v) + k\,Ext^1(M_w, M_v)$. Using properties of simple-minded collections $dim Hom(M_w, M_v)  = k$. Using Prop 6.4 in \cite{KQ15} we can see that $Ext^1(M_w, M_v)  = 0$, Hence $\langle w', w'\rangle = 1$. $w'$ is real.\\
\indent The assumption has been refuted. Any $c$-vector $c$ that appears in any MGS is a real Schur root.
\end{proof}
\indent There is also a result we will need to use to prove \ref{T1}.
\begin{lemma}\label{L}
\indent If $-c_1, -c_2$ are negative $c$-vectors in $C$-matrix $C'$ in an MGS, $c_1$ and $c_2$ are dimension vectors of exceptional modules $M_1$ and $M_2$. If $dim Ext^1(M_1, M_2) > 1$ then the mutation on $C'$ must not be done on $M_2$.
\end{lemma}
\begin{proof}
Assume that $-c_1$ is the $i$-th column and $-c_2$ is the $j$-th colu,n. According to \cite{KY12} using the definition of left mutations of simple-minded collections if $dim Ext^1(M_1, M_2) > 1$ then the mutation on $-c_2$ would cause $-c_1$ to be transformed into $-c_1-kc_2$ with $k>1$ because which could only happen if there are multiple edges from $i$ to $j$. Due to \cite{BHIT15} this was impossible.\\
\end{proof}
\subsection{MGS-finiteness}
\indent In this subsection let's review the basics about what kind of quivers have finitely many maximal green sequences.
\begin{definition}
A quiver $Q$ is \textit{MGS-finite} if $Q$ has finitely many maximal green sequences. Any quiver that isn't MGS-finite is \textit{MGS-infinite}.
\end{definition}
\indent Here are some results that are either already known or easily proven about MGS-finiteness of quivers.
\begin{theorem}
\cite{BDP13}Any acyclic quiver $Q$ of finite type or tame type as well as any acyclic quiver $Q$ of wild type with three vertices are MGS-finite.
\end{theorem}
\begin{theorem}
Any quiver $Q$ mutation equivalent to an acyclic quiver of finite or tame type is MGS-finite.
\end{theorem}
\begin{proof}
Due to \cite{BHIT15} the result is already proven in the mutation-equivalent to tame type case. For the mutation-equivalent to finite type case using the Rotation Lemma in \cite{BHIT15} it is obvious that any MGS in such a quiver must be an $k$-reddening sequence of an acyclic quiver of finite type for a fixed $k$. There are only finitely many such sequences because a $k$-reddening sequence can only repeat a cluster $k+1$ times as shown in \cite{BHIT15} and \cite{IZ17} and in an acyclic quiver of finite type there are only finitely many cluster-tilting objects and hence clusters.
\end{proof}
\begin{lemma}
If $Q$ is a quiver that isn't connected, $Q^1$, $Q^2$, $\cdots$ $Q^n$ are its connected components. Each $Q^i$ is MGS-finite if and only if $Q$ is MGS-finite.
\end{lemma}
\begin{proof}
\indent Any MGS of $Q$ is essentially formed from taking an MGS $w_i$ of $Q^i$ for each $i$ and then put these mutations together such that the order of elements in each $w_i$ is preserved.\\ 
\indent Since we can obtain all MGSs of $Q^i$ by deleting all $c$-vectors not supported on $Q^i_0$ from all MGSs of $Q$ it is easy to see that if $Q$ is MGS-finite so is $Q^i$ for any $i$.\\
\indent On the other hand if all $Q^i$s are MGS-finite it is easy to see that so is $Q$ because the set of admissible $c$-vectors of $Q$ is the union of admissible $c$-vectors in MGSs of $Q^i$ all of which are finite.\\
\end{proof}
\section{The acyclic case}
\indent Now we need some basic definitions in order to describe and prove the results.
\begin{definition}
A quiver with at least one multiple edge is \textit{ME-ful}. Otherwise it is \textit{ME-free}.
\end{definition}
\begin{definition}
A \textit{multiple edges-free (ME-free)} version of a quiver $Q$ is produced by removing all multiple edges from $Q$ while retaining single edges and vertices.
\end{definition}
\indent For example the ME-free version of the $m$-Kronecker quiver for any $m$ is the quiver $A_1\times A_1$, namely the quiver with two vertices and no arrows.\\
\indent In this section we will use the fact that a path in the semi-invariant picture of $Q$ is also a path in the semi-invariant picture of its ME-free version, $Q'$. Since the definition of whether a path is green and generic differ in semi-invariant pictures of different quivers we will use the concept of \textit{strong genetic green paths} to exclude problematic cases.
\begin{definition}
Let $Q$ be an ME-ful quiver, $Q'$ be its ME-free version. A path in the semi-invariant pictures of $Q$ and $Q'$ is \textit{strong generic green} if it is a generic green path in both pictures.
\end{definition}
\begin{definition}
Let $Q$ be an ME-ful quiver.
\begin{enumerate}
\item A $c$-vector in $Q$ is \textit{ME-free} if its support is ME-free. Any $c$-vector in $Q$ that isn't ME-free is \textit{ME-ful}.
\item An MGS in $Q$ is \textit{ME-free} if all its $c$-vectors are ME-free. An MGS of $Q$ that isn't ME-free is \textit{ME-ful}.
\item A generic green path in the semi-invariant picture of $Q$ is \textit{ME-free} if it crosses no wall corresponding to an ME-ful $c$-vector. A generic green path in the semi-invariant picture of $Q$ that isn't ME-free is \textit{ME-ful}.
\item A module of $kQ$ is \textit{ME-free/ME-ful} if its $c$-vector is ME-free/ME-ful.
\end{enumerate}
\end{definition}
\indent Note that if an MGS is ME-free all $c$-vectors in all $c$-matrices in it including those that aren't mutated must be ME-free.\\
\indent If $Q$ is an ME-ful quiver and $Q'$ is its ME-free version it does not technically make sense to discuss ME-fulness of any module of $kQ'$. Here we are going to use the same definition we used in defining ME-fulness of vectors and MGSs of $Q$.
\begin{definition}
Let $Q$ be an ME-ful quiver and let $Q'$ be its ME-free version.
\begin{enumerate}
\item A $c$-vector in $Q'$ is \textit{Q-ME-free} if it is ME-free when considered as a dimension vector of $Q$. Any $c$-vector in $Q'$ that isn't $Q$-ME-free is \textit{Q-ME-ful}.
\item An MGS in $Q'$ is \textit{Q-ME-free} if all its $c$-vectors are $Q$-ME-free. An MGS of $Q'$ that isn't $Q$-ME-free is \textit{Q-ME-ful}.
\item A generic green path in the semi-invariant picture of $Q'$ is \textit{Q-ME-free} if it crosses no wall corresponding to a $Q$-ME-ful $c$-vector. A generic green path in the semi-invariant picture of $Q'$ that isn't $Q$-ME-free is \textit{Q-ME-ful}.
\item A strongly generic green path in the semi-invariant picture of $Q'$ is \textit{strongly $Q$-ME-free} if it is $Q$-ME-free and does not cross any wall corresponding to a $Q$-ME-ful $c$-vector in the semi-invariant picture of $Q$. A generic green path in the semi-invariant picture of $Q'$ that isn't strongly $Q$-ME-free is \textit{weakly Q-ME-ful}.
\item A module of $kQ'$ is \textit{Q-ME-free/Q-ME-ful} if its $c$-vector is $Q$-ME-free/$Q$-ME-ful.
\end{enumerate}
\end{definition}
\indent We will sometimes abuse the notations and use the term $Q$-ME-free for $c$-vectors/MGSs of $Q$. In this case they are just ME-free $c$-vectors/MGSs.
\begin{definition}
If $Q$ and $Q'$ have the same number of vertices, a GS $w$ of $kQ$ \textit{is equivalent to} a GS $w'$ of $kQ'$ if $w$ and $w'$ mutates on the same sequence of $c$-vectors and start from the same $c$-matrix up to permutations. 
\end{definition}
\indent Using the equivalence it makes sense to identify certain MGSs of $Q$ and $Q'$. It is in this sense that we claim and prove that all MGSs of an ME-ful quiver $Q$ are MGSs of its ME-free version, $Q'$.
\indent In order to state a corollary we also need three more definitions.
\begin{definition}
The \textit{skeleton} of a quiver $Q$ is produced by replacing all multiple edges from $Q$ by single edges with the sources and targets unchanged.
\end{definition}
\indent For example the ME-free version of the $m$-Kronecker quiver for any $m$ is the quiver $A_2$.
\begin{definition}
$Q$ and $Q'$ are quivers. If they have the same ME-free version and the same skeleton then they are \textit{ME-equivalent}.
\end{definition}
\begin{definition}
If every MGS of $Q$ corresponds to some MGS of $Q'$ and vice versa then $Q$ and $Q'$ are MGS-equivalent. 
\end{definition}
\begin{lemma}
\indent Let $Q$ be a quiver and $Q'$ be its ME-free version. The following holds:\label{L2}
\begin{enumerate}
\item The set of $Q$-ME-free $c$-vectors of $Q$ and $Q'$ coincide.
\item If $Q$ is an ME-ful quiver then for any positive $Q$-ME-ful vector $c\in\mathbb{R}^n$ $\langle M,M\rangle_{kQ} - \langle M,M\rangle_{kQ'} \leq -2$.
\item If $Q$ is an ME-ful quiver. Then any of the $Q$-ME-ful $c$-vectors can not be a dimension vector of an exceptional modue for $Q'$. Any of the $Q$-ME-ful $c$-vectors of $Q'$ can not be a dimension vector of an exceptional module for $Q$.
\end{enumerate}
\end{lemma}
\begin{proof}
\indent For (1).Let the Euler matrices of $Q, Q''$ be $E = e_{ij}, E' = (e'_{ij})$ respectively. $\langle c,c\rangle_{kQ} = \langle c,c\rangle_{kQ'}$ because whenever $e_{ij}, e'_{ij}$ differ $c_i = 0$ or $c_j = 0$ leaving the term related to $(i,j)$ 0. Hence the set of $Q$-ME-free $c$-vectors of $Q, Q'$ corresponding to exceptional modules coincide.\\
\indent For (2) Assume that such a vector, $c$ exists. $\langle c,c\rangle_{kQ} = \langle c,c\rangle_{kQ'} = 1$. However the Euler matrix $E = (e_{ij})$ of $Q$ and the Euler matrix $E' = (e'_{ij})$ of $Q'$ differ in the sense that there exists some pair $(i,j)\in [n]$ such that $c_i\neq 0, c_j> 0$ and $0 = e'_{ij} > -2 \geq e_{ij}$. Since for any $k,l\in [n]$ $e'_{kl}\geq e_{kl}$ it is easy to see that $\langle c,c\rangle_{kQ'} > \langle c,c\rangle_{kQ}$ and that $\langle M,M\rangle_{kQ} - \langle M,M\rangle_{kQ'} \leq -2$.\\
\indent (3) is a consequence of (2) since $\langle M,M\rangle_{kQ}$ and $\langle M,M\rangle_{kQ'} $ can not both be 1.
\end{proof}
\begin{lemma}
Let $Q$ be an ME-ful quiver. Any MGS of an ME-ful quiver $Q$ must not contain any $Q$-ME-ful $c$-vector of $Q'$ or any vector $c$ which is an imaginary root of $Q'$.
\end{lemma}
\begin{proof}
\indent Due to \ref{L2}(3) we only need to prove the second part. In that case $\langle c,c\rangle_{kQ} \leq \langle c,c\rangle_{kQ'} < 1$. Hence $c$ is not a $c$-vector of $Q$.
\end{proof}
\indent Now we can easily establish the following theorem.
\begin{theorem}
MGSs of an acyclic quiver $Q$ are a subset of the set of $Q$-ME-free MGSs of its ME-free version, $Q'$.\label{T1B}
\end{theorem}
\begin{proof}
\indent If the statement is incorrect along an MGS of $Q$ pick the first $C$-matrix that isn't shared by $Q'$ assuming that such an MGS exists. \\
\indent In this case either at least one $c$-vector is $Q$-ME-ful or none is. If some $c$-vector is $Q$-ME-ful it must be formed by extending one $Q$-ME-free exceptional module by another $Q$-ME-free exceptional module in $Q$ (i.e. $dim Ext_{kQ}(A,B) = 1$ because it can not be larger due to Lemma \ref{L}. Let's label the indecomposable module formed by the extension $M$. We need $Ext_{kQ}(B,A) = 0$ so that $\langle M,M\rangle_{kQ} = 1$) while in $Q'$  there are no such extensions (i.e. $dim Ext_{kQ'}(A,B) = 0$). However this is impossible because $A, B$ are rigid, $Hom$-orthogonal and indecomposable because $\langle M,M\rangle_{kQ} - \langle M,M\rangle_{kQ'} \leq -2$ which causes $\langle M,M\rangle_{kQ'}$ to be at least 3 which is impossible because $\langle M,M\rangle_{kQ'} = \langle A,A\rangle_{kQ'} + \langle A,B\rangle_{kQ'} + \langle B,A\rangle_{kQ'} + \langle B,B\rangle_{kQ'} = 2 - dim Ext_{kQ'}(A,B) - dim Ext_{kQ'}(B,A)$ is at most 2.\\
\indent If no $c$-vector is $Q$-ME-ful then in $kQ$, $kQ'$ the relevant $Hom$ and $Ext$ groups shouldn't differ because neither of them involve the multiple edges that are absent in $kQ'$ . As a result that can't happen either.\\
\indent Hence the $C$-matrices corresponding to $Q, Q'$ in the MGS are all the same. Any MGS of $Q$ must be an MGS of $Q'$ with the same $C$-matrices. Since all the $C$-matrices of the two quivers are the same they have the same associated permutation.\\
\end{proof}
\indent Now we can prove a stronger result. In order to do so we will use a sequence of $C$-matrices as MGSs.\\
\begin{theorem}\label{T3B}
Let $Q$ be an ME-ful acyclic quiver and $Q'$ be its ME-free version. The MGSs of $Q$ are exactly the $Q$-ME-free MGSs $(C_0,C_1,\cdots C_m)$ of $Q'$ such that for any multiple edge from $i$ to $j$ in $Q$ for any $C$-matrix $C_i$ in the MGS such that there exists a negative $c$-vector with support containing $i$ the mutation on $C_i$ in the MGS isn't done on any negative $c$-vector with support containing $j$.
\end{theorem}
\begin{proof}
Let's compare $\langle M, N \rangle_{kQ}$ and $\langle M, N \rangle_{kQ'}$. They differ if and only if there exists some multiple edge from $i$ to $j$ such that $i$ is in the support of $M$ and $j$ is in the support of $N$. In this case since $Hom_{kQ}(M,N) = Hom_{kQ'}(M,N)  = Ext_{kQ'}(M,N) = 0$ $dim Ext_{kQ}(M,N) > 0$. Repeating the argument in \ref{T1B} we can show that this is the only possible scenario for a $Q$-ME-free MGSs of $Q'$ to not be identical to an MGS in $Q$.
\end{proof}
\begin{corollary}\label{CB}
The following statements are true:
\begin{enumerate}
\item The number of maximal green sequences of a quiver $Q$ is no greater than that of its ME-free version.
\item All quivers with an MGS-finite ME-free version must themselves be MGS-finite.
\item No minimally MGS-infinite quiver can contain multiple edges.
\item Any two ME-equivalent quivers are MGS-equivalent to each other.
\end{enumerate}
\end{corollary}
\begin{proof}
\indent Only (4) needs to be proven even though it is still obvious. For ME-equivalent quivers $Q$ and $Q'$ the conditions of \ref{T3B} are identical which is why the number of MGS are identical.
\end{proof}
\begin{example}
The maximal green sequences of $Q: 
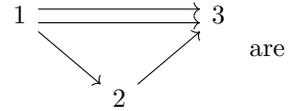
\begin{tikzcd}
1\righttwicedoublearrow\arrow[rd] &  & 3\\
 & 2\arrow[ur]
\end{tikzcd}$ are maximal green sequences of its ME-free version $Q': 1\to 2\to 3$ that has no $c$-vector with support containing $\{1,3\}$ and satisfies the conditions in \ref{T3B} with respect to the arrow $\begin{tikzcd} 1\rightdoublearrow & 3\end{tikzcd}$. It's easy to see that $Q$ is MGS-finite. In fact it has 3 MGSs.
\end{example}
\begin{example}
The maximal green sequences of $Q: \begin{tikzcd}
1\arrow[r] & 2 \rightdoublearrow&  3 \arrow[r] & 4\\
\end{tikzcd}$ are some maximal green sequences of its ME-free version $Q': \begin{tikzcd}[cramped, sep=small]
1\arrow[r] & 2 & 3 \arrow[r] & 4\\
\end{tikzcd}$ that has no $c$-vector with support containing $\{2,3\}$ and satisfies the conditions in \ref{T3B} with respect to the arrow $\begin{tikzcd} 2\rightdoublearrow & 3\end{tikzcd}$. It's easy to see that $Q$ is MGS-finite because $A_2$ is.
\end{example}
\indent Now we can provide a much shorter proof to the fact that all acyclic quivers with three vertices are MGS-finite which was originally proven in \cite{BDP13}.
\begin{corollary}
Any acyclic quiver with at most three vertices is MGS-finite.
\end{corollary}
\begin{proof}
Due to the theorem we only need to show that any ME-free acyclic quiver with at most three vertices is MGS-finite. Such a quiver is either of finite or tame type and is hence MGS-finite.
\end{proof}
\section{The general case}
\indent In the general case the theorems above aren't true. We can show that using the following counterexample. The quiver $Q$ here is $
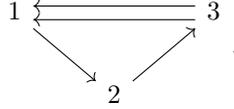
\begin{tikzcd}
1\arrow[rd] &  & 3\lefttwicedoublearrow\\
 & 2\arrow[ur]
\end{tikzcd}$.\\
$\begin{bmatrix} 
0 &1 & -2\\
-1 & 0 & 1\\
2 & -1 & 0\\
-1 & 0 & 0\\
0 & -1 & 0\\
0 & 0 & -1\\
\end{bmatrix}\overset{\mu_2}{\to}\begin{bmatrix} 
0 &-1 & -1\\
1 & 0 & -1\\
1 & 1 & 0\\
-1 & 0 & 0\\
-1 & 1 & 0\\
0 & 0 & -1\\
\end{bmatrix}\overset{\mu_1}{\to}\begin{bmatrix} 
0 &1 & 1\\
-1 & 0 & -1\\
-1 & 1 & 0\\
1 & -1 & -1\\
1 & 0 & -1\\
0 & 0 & -1\\
\end{bmatrix}\overset{\mu_3}{\to}\begin{bmatrix} 
0 &2 & -1\\
-2 & 0 & 1\\
1 & -1 & 0\\
0 & -1 & 1\\
0 & 0 & 1\\
-1 & 0 & 1\\
\end{bmatrix}\overset{\mu_1}{\to}\begin{bmatrix} 
0 &-2 & 1\\
2 & 0 & -1\\
-1 & 1 & 0\\
0 & -1 & 1\\
0 & 0 & 1\\
1 & 0 & 0\\
\end{bmatrix}\overset{\mu_2}{\to}\begin{bmatrix} 
0 &2 & -1\\
-2 & 0 & 1\\
1 & -1 & 0\\
0 & 1 & 0\\
0 & 0 & 1\\
1 & 0 & 0\\
\end{bmatrix}$\\
\indent Here we have a maximal green sequence with at least one ME-full $c$-vector. Moreover it is easy to see that if we replace the double edge by triple edge and obtain $Q': 
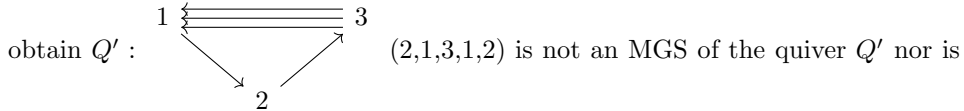
\begin{tikzcd}
1\arrow[rd] &  & 3\lefttwicetriplearrow\\
 & 2\arrow[ur]
\end{tikzcd}$\ (2,1,3,1,2) is not an MGS of the quiver $Q'$ nor is it an MGS of the ME-free version or skeleton of $Q$.\\
\indent However we can still perform quiver cutting in more limited situations. Let's first introduce a concept.
\begin{definition}
A $k$\textit{-edge} is a tuple $(i,j)$ where $i,j\in [n]$ and $k|b_{ij}, k|b_{ji}$.
\end{definition}
\begin{definition}
Let $Q$ be a quiver possibly having oriented cycles, let $k$ be an integer greater than 1. Assume that $Q_0$ = $\tilde{Q}_0 + \breve{Q}_0$, $P = Q]_{\tilde{Q}_0}, R = Q]_{\breve{Q}_0}$. If for all $i\in \tilde{Q}_0, j\in  \breve{Q}_0$ $k|b_{ij}$ and $k|b_{ji}$ we say $Q$ is $k$-\textit{partible} and $(\tilde{Q}, \breve{Q})$ is a $k$\textit{-partition} of $Q$.
\end{definition}
\begin{theorem}
Assume that ($\tilde{Q},\breve{Q})$ are $k$-partition of $Q$ for some $k>1$ any MGS of $Q$ is an MGS of $\tilde{Q}\cup\breve{Q}$.\label{T2B}
\end{theorem}
\begin{proof}
The property that for any $i\in Q_1$ and $j\in Q_2$ $k|c_{ij}$ is preserved by mutation. Hence any mutation that cause any $c$-vector to cross bot has to violate the Sink before Source Theorem.
\end{proof}
\begin{corollary}
Under the conditions of the theorem above, if $\tilde{Q}$ and $\breve{Q}$ are MGS-finite so is $Q$.
\end{corollary}
\begin{proof}
If $\tilde{Q}$ and $\breve{Q}$ are MGS-finite so is $\tilde{Q}\cup\breve{Q}$. As a result so is $Q$ due to the theorem.
\end{proof}
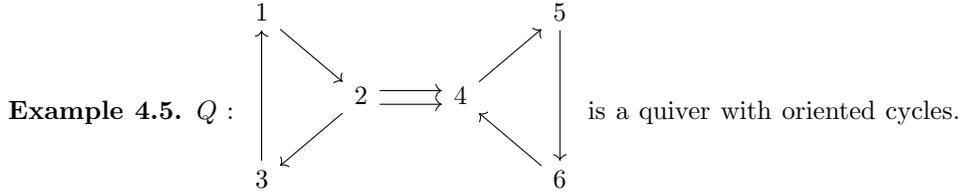
\begin{example}
$Q:\begin{tikzcd}
1\arrow[rd] &                                             &                   &5\arrow[dd]\\
                 &2\arrow[dl]\rightdoublearrow & 4\arrow[ur] &\\
3\arrow[uu] &                                            &                  &6\arrow[ul]\\
\end{tikzcd}$ is a quiver with oriented cycles. Due to the theorem we can cut the $\begin{tikzcd}2\rightdoublearrow & 4\end{tikzcd}$ arrow. After cutting this arrow it is easy to see that $Q$ is MGS-finite.
\end{example}
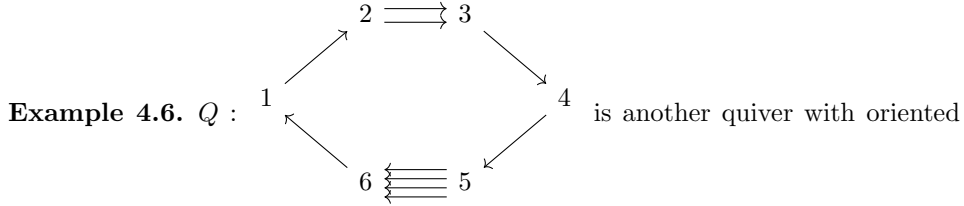
\begin{example}
$Q:\begin{tikzcd}
 &       2\rightdoublearrow                                      &  3\arrow[rd]&\\
1\arrow[ru]                 &    &  &4\arrow[dl]\\
 &     6\arrow[ul]                                       & 5\leftquadruplearrow                 &\\
\end{tikzcd}$ is another quiver with oriented cycles.  Due to the theorem we can cut the $\begin{tikzcd}2\rightdoublearrow & 3\end{tikzcd}$ and  $\begin{tikzcd}6 & 5\leftquadruplearrow\end{tikzcd}$ arrows. After cutting these arrows it is easy to see that $Q$ is MGS-finite.
\end{example}
\bibliographystyle{amsplain}

\begin{thebibliography}{10}

\bibitem{BDP13} Thomas Br\"ustle, Gr\'{e}goire Dupont and Matthieu P\'{e}rotin, \textit{On Maximal Green Sequences},  Int Math Res Notices (2014), 4547--4586.
\bibitem{BHIT15} Thomas Br\"ustle, Stephen Hermes, Kiyoshi Igusa and Gordana Todorov, \textit{Semi-invariant pictures and two conjectures on maximal green sequences}, J Algebra {\bf 473} (2017): 80--109.
\bibitem{BST17} Thomas Br\"ustle, David Smith and Hipolito Treffinger, \textit{Stability conditions, $\tau$-tilting Theory and Maximal Green Sequences}, \href{https://arxiv.org/abs/1705.08227}{arXiv:1705.08227 [math.RT]}, 2017.
\bibitem{FZ01} Sergey Fomin and Andrei Zelevinsky, \textit{Cluster algebras I: Foundations}, J. Amer. Math. Soc. 15 (2002), 497-529.
\bibitem{FZ06} Sergey Fomin and Andrei Zelevinsky, \textit{Cluster algebras IV: Coefficients}, Compositio Math. 143 (2007) 112?164.
\bibitem{GHKK14} Mark Gross, Paul Hacking, Sean Keel and Maxim Kontsevich, \textit{
Canonical bases for cluster algebras}, J. Amer. Math. Soc. 31 (2018), 497-608.
\bibitem{GM14} Alexander Garver and Gregg Musiker, \textit{On Maximal Green Sequences For Type A Quivers}, G. J Algebr Comb (2017) 45: 553.
\bibitem{IOTW15} Kiyoshi Igusa, Kent Orr, Gordana Todorov and Jerzy Weyman, \textit{Modulated semi-invariants},  \href{http://arxiv.org/abs/1507.03051}{arXiv:1507.03051 [math.RT]}.
\bibitem{I17} Kiyoshi Igusa, \textit{Linearity of stability conditions}, \href{https://arxiv.org/abs/1706.06986}{arXiv:1706.06986}
\bibitem{IZ17} Kiyoshi Igusa and Ying Zhou, \textit{Tame Hereditary Algebras have finitely many m-Maximal Green Sequences}, \href{https://arxiv.org/abs/1706.09118}{arXiv:1706.09118}
\bibitem{Kel11} Bernhard Keller, \textit{Quiver mutation and quantum dilogarithm identities}, Representations of Algebras and Related Topics, Editors A. Skowronski and K. Yamagata, EMS Series of Congress Reports, European Mathematical Society (2011): 85-116.
\bibitem{KQ15} Alastair King, Yu Qiu, \textit{Exchange graphs and Ext quivers}, Advances in Mathematics, Volume 285, 1106-1154, 2015.
\bibitem{KY12} Steffen Koenig and Dong Yang, \textit{Silting objects, simple-minded collections, t-structures and co-t-structures for finite-dimensional algebras}, Doc. Math. 19 (2014), 403-438. 

\end{thebibliography}

\end{document}